\documentclass[12pt]{amsart}

\usepackage{hyperref}
\usepackage{amssymb, amsthm, enumerate, mathptmx}

\usepackage{cleveref}

\newtheorem{thm}{Theorem}[section]
\newtheorem{lem}[thm]{Lemma}
\newtheorem{cor}[thm]{Corollary}
\newtheorem{prop}[thm]{Proposition}

\newtheorem{question}[thm]{Question}
\newtheorem{conj}[thm]{Conjecture}
\newtheorem{prob}[thm]{Problem}

\theoremstyle{definition} 
\newtheorem{defn}[thm]{Definition} 
\newtheorem{remark}[thm]{Remark}
\newtheorem{example}[thm]{Example}
\newtheorem{convention}[thm]{Convention}

\DeclareMathOperator{\codim}{codim}
\DeclareMathOperator{\pd}{pd}

\DeclareMathOperator{\Sym}{Sym}
\DeclareMathOperator{\Inc}{Inc}

\DeclareMathOperator{\ind}{ind}
\DeclareMathOperator{\ini}{in}

\DeclareMathOperator{\Min}{Min}

\def\Icc{{\mathcal I}}

\newcommand{\N}{\mathbb{N}}
\newcommand{\Z}{\mathbb{Z}}
\newcommand{\be}{{\bf e}}

\title{Codimension and Projective Dimension up to Symmetry}

\author{Dinh Van Le}
\address{Institut f\"ur Mathematik, Universit\"at Osnabr\"uck, 49069 Osnabr\"uck, Germany}
\email{dlevan@uos.de}

\author{Uwe Nagel}
\address{Department of Mathematics, University of Kentucky, 715 Patterson office tower, Lexington, KY 40506-0027, USA}
\email{uwe.nagel@uky.edu}

\author{Hop D. Nguyen}
\address{Institute of Mathematics, Vietnam Academy of Science and Technology, 18 Hoang Quoc Viet, 10307 Hanoi, Vietnam}
\email{ngdhop@gmail.com}

\author{Tim R\"{o}mer}
\address{Institut f\"ur Mathematik, Universit\"at Osnabr\"uck, 49069 Osnabr\"uck, Germany}
\email{troemer@uos.de}

\begin{document}

\begin{abstract}
Symmetric ideals in increasingly larger polynomial rings that form an ascending chain are investigated. We focus on the asymptotic behavior of codimensions and projective dimensions  of ideals in such a chain. If the ideals are graded it is known that the codimensions grow eventually linearly. Here this result is extended to chains of arbitrary symmetric ideals. Moreover, the slope of the linear function is explicitly determined. We conjecture that the projective dimensions  also grow eventually linearly. As part of the evidence we establish two non-trivial lower linear bounds of the projective dimensions for chains of monomial ideals. As an application, this yields Cohen-Macaulayness obstructions.
\end{abstract}

\thanks{The second author was partially supported by Simons Foundation grant \#317096. The third author was partially supported by a postdoctoral fellowship from the Vietnam Institute for Advanced Study in Mathematics (VIASM). Nguyen was also partially supported by Project ICRTM01$\_$2019.01 of the International Centre for Research and Postgraduate Training in Mathematics (ICRTM), Institute of Mathematics, VAST}


\maketitle

\section{Introduction}

Ascending chains of ideals that are invariant under actions of symmetric groups have recently attracted considerable attention. They arise naturally in various areas of mathematics, such as
algebraic chemistry \cite{AH07, Dr10},
group theory \cite{Co67},
representation theory \cite{CEF, NR17+, PS, SS-14,SS-16},
toric algebra and algebraic statistics \cite{BD11,Dr14,DEKL15,DK14,HM13,HS12,HS07,SS03},
which provide frameworks and motivations for further studies. In \cite{LNNR} we investigated the behavior of the Castelnuovo-Mumford regularity along graded ideals in such a chain. Here we study the analogous problem for codimension and projective dimension.

Let $\N$ denote the set of positive integers. Throughout the paper, fix an integer $c\in \N$ and any field $K$. For each $n\in \N$, let
\[
 R_n=K[x_{k,j}\mid 1\le k\le c,1\le j\le n]
\]
be the polynomial ring in $c\times n$ variables over $K$. These form an ascending chain
\[
R_1\subseteq R_2\subseteq\cdots\subseteq R_n\subseteq\cdots.
\]
Let $\Sym(n)$ denote the symmetric group on $\{1,\dots,n\}$. Considering it as stabilizer of $n+1$ in $\Sym (n+1)$, similarly one gets an ascending chain of symmetric groups. Define an action of $\Sym(n)$ on $R_n$ induced by
\[
 \sigma \cdot x_{k,j}=x_{k,\sigma(j)} \quad\text{for every}\ \sigma\in \Sym(n),\ 1\le k\le c,1\le j\le n.
\]
A sequence of ideals $(I_n)_{n\ge1}$ with $I_n\subseteq R_n$ is called \emph{$\Sym$-invariant} if
\[
 \Sym(n)(I_m)\subseteq I_n\quad\text{for all }\ m\le n.
\]
Observe that these ideals form an ascending chain as $I_n \cdot R_{n+1} \subset I_{n+1}$.

Even if one is primarily interested in $\Sym$-invariant chains  it is more convenient to work with a larger class of invariant objects, namely $\Inc^i$-invariant chains, where $\Inc^i$ denotes a certain monoid of increasing functions on $\N$ (see \Cref{sec2} for more details).

In \cite{NR17}, the second and fourth author introduced  Hilbert series for $\Inc^i$-invariant chains and proved that these series are rational (see also \cite[Theorem 4.3]{KLS} for another approach and \cite[Theorems 2.4 and 3.3]{GN} for some explicit results in a special case). As a consequence, they determined the asymptotic behavior of the Krull dimension and multiplicity of graded ideals in an $\Inc^i$-invariant chain: the Krull dimension grows eventually linearly, whereas the multiplicity grows eventually exponentially. This result motivates a more general line of investigations:

\begin{prob}
 Study the asymptotic behavior of invariants of ideals in $\Sym$-invariant or, more generally, $\Inc^i$-invariant chains.
\end{prob}

In \cite{LNNR}, this problem was studied in the case of the Castelnuovo-Mumford regularity. There we conjectured that this invariant grows eventually linearly and provided some evidence supporting this conjecture. In particular, a linear upper bound for the Castelnuovo-Mumford regularity of graded ideals was established.
As mentioned above, the present work studies the asymptotic behavior of codimensions (i.e.\ heights) and projective dimensions of ideals in $\Inc^i$-invariant chains.

For an $\Inc^i$-invariant chain of graded ideals $(I_n)_{n\ge 1}$, it follows from \cite[Theorem 7.10]{NR17} that $\codim I_n$ is eventually a linear function. However, not much is known about this function. Here, we extend this result to $\Inc^i$-invariant chains of ideals that are not necessarily graded. More importantly, our new approach also produces an explicit description for the leading coefficient of the linear function (see \Cref{codim}).

To motivate our study on the asymptotic behavior of the projective dimension, let us consider a simple example.

\begin{example}
\label{ex1}
Let $(I_n)_{n\ge 1}$ be an $\Inc^1$-invariant chain with
 \[
  I_n=\begin{cases}
       \langle0\rangle & \text{if } \ n=1,2,3,\\
       \langle x_{1,2}^3,x_{1,4}^2x_{2,1},x_{2,2}x_{3,3}\rangle& \text{if } \ n=4,\\
       \langle\Inc^1_{4,n}(I_4)\rangle& \text{if } \ n\ge5
      \end{cases}
 \]
(see \Cref{ex2} for an explicit description of the ideals in this chain). Computations with Macaulay2 \cite{GS} yield the following table:

\begin{table}[h]
\begin{tabular}{|c|c|c|c|c|c|c|c|}
\hline
\multicolumn{1}{|c|}{$n$} & 4 & 5 & 6 & 7  & 8  & 9  & 10 \\ \hline
$\pd(R_n/I_n)$            & 3 & 6 & 8 & 10 & 12 & 14 & 16 \\ \hline
\end{tabular}
\end{table}

\noindent This table suggests that $\pd(R_n/I_n)$ could be a linear function with slope $2$ when $n\ge5$. 
\end{example}

The previous example and many other computational experiments lead us to the following expectation:

\begin{conj}
\label{conj}
 Let $(I_n)_{n\ge 1}$ be a $\Sym$-invariant or,  more generally, an  $\Inc^i$-invariant chain  of  ideals. Then $\pd (R_n/I_n)$ is eventually a linear function, that is,
\[
\pd (R_n/I_n) =an+b \quad \text{for some integer constants } a,\; b\ \text{ whenever }\ n\gg 0.
\]
\end{conj}

It should be noted that this conjecture as well as \cite[Conjecture 1.1]{LNNR} is seemingly of parallel nature to the well-known asymptotic linearity of the Castelnuovo-Mumford regularity of the powers of a graded ideal shown independently by Cutkosky, Herzog, and Trung \cite[Theorem 1.1(ii)]{CHT} and Kodiyalam \cite[Theorem 5]{K}.

Since $\codim I_n$ is eventually a linear function, \Cref{conj} is clearly true if $I_n$ is perfect for $n\gg 0$ (see \Cref{CM}). Note that, for example, a graded Cohen-Macaulay ideal is perfect.

It is not hard to give linear upper and lower bounds for $\pd (R_n/I_n)$ (see \Cref{pd both bounds}) because
\[
 cn\ge \pd (R_n/I_n) \ge \codim I_n.
\]
Our next main results establish improved lower linear bounds for $\pd (R_n/I_n)$ in the case of a chain of monomial ideals (see Theorems \ref{pd_bound} and \ref{pd bound 2}). These also yield necessary conditions for the Cohen-Macaulayness of $R_n/I_n$ when $n \gg 0$ (see Corollaries \ref{cor:CM condition} and \ref{cor:CM obsrruction}). Note that by the Auslander-Buchsbaum formula all statements on projective dimensions of graded ideals can equivalently be stated as results on depths. 

The paper is divided into four sections. \Cref{sec2} contains some basic notions and facts on invariant chains of ideals. The asymptotic behavior of codimensions and projective dimensions of ideals in such chains are discussed in \Cref{sec3,sec4}, respectively.

{\bf Acknowledgement.} We are grateful to the anonymous referees for insightful comments and suggestions that helped a lot to improve the clarity of the paper. The first author wishes to thank Lorenzo Venturello for his help with computations using Macaulay2.


\section{Preliminaries}\label{sec2}

We keep the notation and definitions of the introduction. In particular, $c$ is a fixed positive integer, and for each $n\ge 1$, $R_n$ denotes the polynomial ring in $c\times n$ variables over a field $K$. Let
\[
 R=\bigcup_{n\ge1}R_n=K[x_{k,j}\mid 1\le k\le c, j\ge 1]
\]
be the polynomial ring in ``$c\times\N$'' variables. The action of $\Sym(n)$ on $R_n$  given by
\[
 \sigma \cdot x_{k,j}=x_{k,\sigma(j)} \quad\text{for all }\ \sigma\in \Sym(n),\ 1\le k\le c,1\le j\le n,
\]
clearly induces an action of 
$$\Sym(\infty)=\bigcup_{n\ge1}\Sym(n)$$ 
on $R$. Recall that a chain of ideals $(I_n)_{n\ge 1}$ with $I_n\subseteq R_n$ is {$\Sym$-invariant} (or {$\Sym(\infty)$-invariant}) if
\[
 \Sym(n)(I_m)=\{\sigma(f)\mid f\in I_m,\ \sigma\in\Sym(n)\}\subseteq I_n\quad\text{for all }\ m\le n.
\]

Often, it is inconvenient to work with $\Sym$-invariant chains. The main reason
is that the group $\Sym(\infty)$ is not compatible  with  monomial orders on $R$; see \cite[Remark 2.1]{BD11}. In particular, the initial chain $(\ini_{\le}(I_n))_{n\ge 1}$ of a $\Sym$-invariant chain $(I_n)_{n\ge 1}$ is typically not $\Sym$-invariant (see \Cref{ex3}). 

To overcome this difficulty, one considers the following monoid of increasing functions on $\N$:
\[
 \Inc = \{\pi \colon \N \to \N \mid \pi(j)<\pi(j+1)\ \text{ for all }\ j\ge 1\},
\]
and more generally, submonoids of $\Inc$ that fix initial segments of $\N$:
\[
\Inc^i=\{\pi \in \Inc \mid \pi(j)=j\ \text{ for all }\ j\le i\},
\]
where $i\ge0$ is an integer. Observe that one has a descending chain of monoids 
\[ 
\Inc=\Inc^0 \supset \Inc^1 \supset \Inc^2 \supset \cdots.
\]

The action of $\Inc^i$ on $R$ is defined analogously to that of $\Sym(\infty)$. We say that
a chain $(I_n)_{n\ge 1}$ with $I_n$ an ideal in $R_n$ is \emph{$\Inc^i$-invariant} if
\[
 \Inc^i_{m,n}(I_m)\subseteq I_n\quad \text{for all }\ m\le n,
\]
where
\[
\Inc^i_{m,n}=\{\pi \in \Inc^i \mid \pi(m)\le n\}.
\]
It is evident that every $\Inc^i$-invariant chain is also $\Inc^{i+1}$-invariant. Moreover, for any $f\in R_m$ and $\pi\in\Inc^i_{m,n}\ (m\le n)$, it is easy to find a permutation $\sigma\in\Sym(n)$ such that $\pi f=\sigma f$ (see, e.g., \cite[Lemma 7.6]{NR17}). Hence, $\Inc^i_{m,n}\cdot f\subseteq\Sym(n) \cdot f$.
It follows that every $\Sym$-invariant chain is also an $\Inc^i$-invariant chain.

A fundamental result of Hillar and Sullivant \cite[Theorem 3.1]{HS12} (see also \cite[Corollary 3.6]{NR17}) implies that every $\Inc^i$-invariant chain $\Icc=(I_n)_{n\ge 1}$ \emph{stabilizes}, meaning that there exists an integer $r\ge1$ such that, as ideals in $R_n$, one has 
\[
I_n=\langle\Inc^i_{r,n}(I_r)\rangle_{R_n} \quad\text{ for all $n\ge r$,}
\]
or equivalently,
\[
I_n=\langle\Inc^i_{m,n}(I_m)\rangle_{R_n} \quad\text{ for all $n\ge m\ge r$}
\]
(see \cite[Lemma 5.2, Corollary 5.4]{NR17}). The least integer $r$ with this property is called the \emph{$i$-stability index} of $\Icc$, denoted by 
\[
\ind^i(\Icc).
\]

\begin{example}
	\label{ex2}
	Let $\Icc=(I_n)_{n\ge 1}$ be the $\Inc^1$-invariant chain considered in \Cref{ex1}. Evidently, $\ind^1(\Icc)=4$. Since $x_{2,1}$ is fixed under the action of $\Inc^1$, some non-zero ideals of $\Icc$ are
	\begin{align*}
		I_4&=\langle x_{1,2}^3,\; x_{1,4}^2x_{2,1},\; x_{2,2}x_{3,3}\rangle, \\
		I_5&=I_4+\langle x_{1,3}^3,\; x_{1,5}^2x_{2,1},\; x_{2,2}x_{3,4},\; x_{2,3}x_{3,4}\rangle, \\
		I_6&=I_5+\langle x_{1,4}^3,\; x_{1,6}^2x_{2,1},\; x_{2,2}x_{3,5},\; x_{2,3}x_{3,5},\; x_{2,4}x_{3,5}\rangle, \\
		I_7&=I_6+\langle x_{1,5}^3,\; x_{1,7}^2x_{2,1},\; x_{2,2}x_{3,6},\; x_{2,3}x_{3,6},\; x_{2,4}x_{3,6},\; x_{2,5}x_{3,6}\rangle.
	\end{align*}
	By induction one can show that for all $n\ge 5$:
	\[
	I_n=I_{n-1}+\langle x_{1,{n-2}}^3,\; x_{1,n}^2x_{2,1},\; x_{2,2}x_{3,n-1},\; x_{2,3}x_{3,n-1},\dots,\; x_{2,n-2}x_{3,n-1}\rangle.
	\]
\end{example}

When working with invariant chains of ideals, a key advantage of the monoids $\Inc^i$ over the group $\Sym(\infty)$ is that the monoids $\Inc^i$ behave well with certain monomial orders on $R$, and the initial chain of any $\Inc^i$-invariant chain with respect to such an order is again $\Inc^i$-invariant (see \Cref{lem_initial_filtration} below). We say that a monomial order $\le$ \emph{respects} $\Inc^i$ if $\pi(u)\le \pi(v)$ whenever $\pi \in \Inc^i$ and $u,v$ are monomials of $R$ with $u\le v$. This condition implies that
\[
\ini_{\le}(\pi(f))=\pi(\ini_{\le}(f))\quad \text{for all } f\in R\ \text{ and } \pi \in \Inc^i.
\]
Examples of  monomial orders respecting $\Inc^i$ include the lexicographic order and the reverse-lexicographic order on $R$ induced by the following ordering of the variables:
\[
x_{k,j}\le x_{k',j'}\quad \text{if either $k<k'$ or $k=k'$ and $j<j'$}.
\]
Throughout this paper, whenever $\le$ is a monomial order on $R$, we will use the same notation to denote its restrictions to the subrings $R_n$.

\begin{example}
	\label{ex3}
	Let $c=1$ and assume that the field $K$ has characteristic $0$. Consider the ideals
	\begin{align*}
		I_3&=\langle x_1^2+x_2x_3,\; x_2^2+x_1x_3,\; x_3^2+x_1x_2\rangle,\\
		I_4&=\Sym(4)(I_3)\\
		&=I_3+\langle x_1^2+x_2x_4,\; x_1^2+x_3x_4,\; x_2^2+x_1x_4,\; x_2^2+x_3x_4,\\
		&\hspace{1.43cm}x_3^2+x_1x_4,\; x_3^2+x_2x_4,\;
		x_4^2+x_1x_2,\; x_4^2+x_1x_3,\; x_4^2+x_2x_3\rangle.
	\end{align*}
	Using the reverse-lexicographic order with $x_1<x_2<x_3<x_4$ one obtains by computations with Macaulay2 that
	\begin{align*}
		\ini(I_3)&=\langle x_2^2, \; x_3x_2,\; x_3^2,\; x_2x_1^2,\; x_3x_1^2,\; x_1^4 \rangle,\\
		\ini(I_4)&=\langle x_2x_1,\; x_3x_1,\; x_4x_1,\; x_2^2,\; x_3x_2,\; x_4x_2,\; x_3^2,\; x_4x_3,\; x_4^2,\; x_1^3 \rangle.
	\end{align*}
	Since $x_2^2\in \ini(I_3)$ but $x_1^2\not\in \ini(I_3),$ we see that $\Sym(3)(\ini(I_3))\nsubseteq\ini(I_3)$, i.e., $\ini(I_3)$ is not $\Sym(3)$-stable. By the same reason,
	$
	\Sym(4)(\ini(I_3))\nsubseteq\ini(I_4).
	$
	Note, however, that $\ini(I_3)$ is trivially $\Inc_{3,3}$-stable, and moreover,
	\[
	\Inc_{3,4}(\ini(I_3))=\ini(I_3)+\langle x_4x_2,\; x_4x_3,\; x_4^2,\; x_4x_1^2 \rangle\subseteq\ini(I_4).
	\]
\end{example}

The phenomenon in the preceding example holds true more generally:

\begin{lem}[{\cite[Lemma 7.1]{NR17}}]
	\label{lem_initial_filtration}
	Let $\Icc=(I_n)_{n\ge 1}$ be an $\Inc^i$-invariant chain of ideals. Then for any monomial order $\le$ respecting $\Inc^i$, the chain $\ini_{\le}(\Icc)=(\ini_{\le}(I_n))_{n\ge 1}$ is also $\Inc^i$-invariant and
	$$
	\ind^i(\Icc) \le \ind^i(\ini_{\le}(\Icc)).
	$$
\end{lem}

We conclude this section with two auxiliary results that will be used frequently. The first one slightly generalizes one part of \cite[Lemma 6.4]{NR17}.

\begin{lem}
\label{colon chain}
Let $\Icc=(I_n)_{n\ge 1}$ be an $\Inc^i$-invariant chain of monomial ideals and $v\in R_i$ a monomial. Then  the chain $\Icc:v=(I_n:v)_{n\ge 1}$ is also $\Inc^i$-invariant and $$\ind^i(\Icc:v)\le \ind^i(\Icc).$$
\end{lem}

\begin{proof}
 Write $v=x_{k_1,j_1}^{e_1}\cdots x_{k_m,j_m}^{e_m}$ with $j_1,\dots,j_m\le i$. By induction on $m$, we may assume $v=x_{k_1,j_1}^{e_1}$. But in this case the result follows by using the same argument as in the proof of \cite[Lemma 6.4]{NR17}.
\end{proof}

For the next result we need further notation.
The \emph{$i$-shift} $\sigma_i\in\Inc^i$ is given by
\[
\sigma_i(j) =
\begin{cases}
j &\text{if }\ 1\le j\le i,\\
j+1 &\text{if }\ j\ge i+1.
\end{cases}
\]
For a graded ideal $J$ in $R_n$ we write $\delta(J)$ for the largest degree of a minimal homogeneous generator of $J$. Note that $\delta(J)$ is well-defined since the degree sequence of any minimal set of homogeneous generators of $J$ is uniquely determined by $J$ (this is true more generally for all graded Betti numbers of $J$; see, e.g., \cite[Proposition 1.5.16]{BH}). Now define the \emph{$q$-invariant} of $J$ as
\[
q(J)=\sum_{j=0}^{\delta(J)}\dim_K (R_n/J)_j.
\]

\begin{lem}
 \label{e chain}
        Let $\Icc=(I_n)_{n\ge 1}$ be an $\Inc^i$-invariant chain of monomial ideals. 
	For each ${\bf e}=(e_1,\ldots,e_c)\in \mathbb{Z}^c_{\ge0}$, consider a chain of monomial ideals $\Icc_{\bf e}=(I_{{\bf e},n})_{n\ge 1}$
	given by
	\[I_{{\bf e},n}=
	\langle (I_n:x_{1,i+1}^{e_1}\cdots x_{c,i+1}^{e_c}),x_{1,i+1},\ldots,x_{c,i+1}\rangle \quad\text{for all }\ n\ge 1.
	\]
	Then the following statements hold:
	\begin{enumerate}
	 \item
	 $\Icc_{\bf e}$ is an $\Inc^{i+1}$-invariant chain with $\ind^{i+1}(\Icc_{\bf e})\le \ind^i(\Icc)+1$.

        \item
	Fix ${\bf e}\in \mathbb{Z}^c_{\ge0}$. Then for every $r\ge \ind^i(\Icc)$ one has
	\[
	 q(I_{{\bf e},r+1})\le q(I_r),
	\]
	and  equality holds if and only if 
	\[
	I_{{\bf e},n+1}=\langle\sigma_i(I_n),x_{1,i+1},\ldots,x_{c,i+1}\rangle\ \text{ and }\  R_{n+1}/I_{{\bf e},n+1} \cong R_n/I_n\quad \text{for all }\ n\ge r.
	\]
	\end{enumerate}
\end{lem}

\begin{proof}
See \cite[Lemma 5.3]{LNNR} (and also the proofs of \cite[Theorem 6.2, Lemma 6.10, Lemma 6.11]{NR17}).
\end{proof}


\section{Codimension up to symmetry}\label{sec3}

Fix a nonnegative integer $i$. From \cite[Theorem 7.10]{NR17} it follows that the codimensions (i.e.\ heights) of graded ideals in an $\Inc^i$-invariant chain grow eventually linearly. 
In this section we extend this result to linearity of the codimension of not necessarily graded ideals in an $\Inc^i$-invariant chain. Moreover, the arguments produce an explicit description for the leading coefficient of the linear function.

We first introduce a function that is used to define that leading coefficient.
Write $[c]=\{1,\dots,c\}$. For a monomial $1 \neq u \in R_n$, let $\min(u)$ (respectively, $\max(u)$) denote the smallest (respectively, largest) index $j$ such that $x_{k,j}$ divides $u$ for some $k\in [c]$. When $J$ is a proper monomial ideal in $R_n$ with minimal set of monomial generators $G(J)$, we set
\begin{align*}
 G_i^+(J)&=\{u\in G(J)\mid \min(u)>i\},\\
 G_i(J)&=\{u\in G(J)\mid \min(u)\le i <\max(u)\},\\
 G_i^-(J)&=\{u\in G(J)\mid \max(u)\le i\}.
\end{align*}

\begin{defn}
	\label{def:i cover}
	Let $C$ be a subset of $[c]$, $u\in R_n$ a monomial, and $J\subsetneq R_n$ a monomial ideal. We say that
	\begin{enumerate}
	 \item
	 $C$ \emph{covers} $u$ if there exists $k\in C$ such that $x_{k,j}$ divides $u$ for some $j\ge 1$,
	
	 \item
	 $C$ is an \emph{$i$-cover} of $J$ if $C$ covers every element of $G_i^+(J)$.
	\end{enumerate}
Let
\[
\gamma_i(J)=\min\{\#C \ | \ C \text{ is an $i$-cover of $J$} \}.
\] 
Note that $0 \le \gamma_i (J) \le c$. If  $J=R_n$, we adopt the convention that
$$
\gamma_i(R_n)=\infty.
$$
\end{defn}

\begin{example}
\label{ex4}
Assume $c\ge 3$ and consider the ideal
\[J=\langle x_{2,1}^4,\; x_{1,1}^3x_{2,3}^2x_{1,4},\; x_{3,2}x_{1,3}^2x_{2,4},\; x_{2,3}^3x_{1,4}^2,\; x_{2,4}^2x_{3,5}^4\rangle \subset R_6.
\]
Then $G_2^-(J)=\{x_{2,1}^4\}$, $G_2(J)=\{x_{1,1}^3x_{2,3}^2x_{1,4},\; x_{3,2}x_{1,3}^2x_{2,4}\}$, $G_2^+(J)=\{x_{2,3}^3x_{1,4}^2,\; x_{2,4}^2x_{3,5}^4\}$.
One sees that $J$ has two minimal 2-covers: $C_1=\{2\}$ and $C_2=\{1,3\}$.
Thus, 
\[
\gamma_2(J)=1.
\]
\end{example}

Let us now discuss some basic properties of the function $\gamma_i$. For a monomial ideal $J$ we first show that $\gamma_i(J)$ can be computed from a primary decomposition of the ideal $\langle G_i^+(J)\rangle$ (or more efficiently, from the minimal primes of $\langle G_i^+(J)\rangle$). Consider the map $\varphi$ that assigns the variable $x_{k,j}$ to $k$ for every $j\ge 1$. Then $\varphi$ clearly induces a map, still denoted by $\varphi$, from the set $\Min(\langle G_i^+(J)\rangle)$ of minimal primes of $\langle G_i^+(J)\rangle$ to the set of $i$-covers of $J$. Let $\mathcal{C}_i(J)$ be the image of this map, i.e.
\[
\mathcal{C}_i(J)=\{\varphi(P)\mid P\in \Min(\langle G_i^+(J)\rangle)\}.
\]

\begin{prop}
\label{min_primes}
	Let $J\subsetneq R_n$ be a monomial ideal. If $C$ is a minimal $i$-cover of $J$, then $C\in\mathcal{C}_i(J)$. Furthermore,
	\[
	\gamma_i(J)=\min\{\#C\mid C\in \mathcal{C}_i(J)\}=\min\{\#\varphi(P)\mid P\in \Min(\langle G_i^+(J)\rangle)\}.
	\]
\end{prop}

\begin{proof}
	The first assertion implies the claimed formula for $\gamma_i(J)$ because it says that $\mathcal{C}_i(J)$, which is a subset of the set of $i$-covers of $J$, contains all minimal $i$-covers. So it suffices to prove this assertion. Suppose $C$ is a minimal $i$-cover of $J$. Then for each $u\in G_i^+(J)$ there exist $k(u)\in C$ and $j(u)\ge 1$ such that $x_{k(u),j(u)}$ divides $u$. Set
	\[
	Q=\langle x_{k(u),j(u)}\mid u\in G_i^+(J)\rangle.
	\]
	Then $Q$ is evidently a prime ideal containing $\langle G_i^+(J)\rangle$. It follows that $Q\supseteq P$ for some $P\in\Min(\langle G_i^+(J)\rangle)$. One has
	\[
	C\supseteq \varphi(Q)\supseteq \varphi(P).
	\]
	Due to the minimality of $C$, this yields $C=\varphi(P)\in\mathcal{C}_i(J)$, because $C$ and $\varphi(P)$ are both $i$-covers of $J$.
\end{proof}

\begin{example}
 Consider again the ideal $J$ in \Cref{ex4}. The set of minimal primes of the ideal $\langle G_2^+(J)\rangle=\langle x_{2,3}^3x_{1,4}^2,\; x_{2,4}^2x_{3,5}^4\rangle$ is
 \[
  \Min(\langle G_2^+(J)\rangle)=\{\langle x_{2,3},\; x_{3,5}\rangle, \langle x_{2,3},\; x_{2,4}\rangle,\langle x_{1,4},\; x_{2,4}\rangle,\langle x_{1,4},\; x_{3,5}\rangle\}.
 \]
Thus,
\[
 \mathcal{C}_2(J)=\{\{2,3\},\{2\},\{1,2\},\{1,3\}\},
\]
and again we find that 
\[
\gamma_2(J)=\min\{\#C\mid C\in \mathcal{C}_2(J)\}=1.
\]
\end{example}

Some further properties of the function $\gamma_i$ are given in the following lemmas.

\begin{lem}
 \label{i cover properties}
 Let $C\subseteq [c]$, let $u,v\in R_n$ be monomials with $u|v$, and let $J\subseteq J'\subsetneq R_n$ be monomial ideals. Then the following statements hold:

\begin{enumerate}
 \item
 If $C$ covers $u$, then $C$ also covers $v$.

 \item
 If $C$ is an $i$-cover of $J'$, then $C$ is also an $i$-cover of $J$.

 \item
 $\gamma_i(J)\le \gamma_i(J')$.

 \item
 $\gamma_i(J)=\gamma_i(\sqrt{J})$.
\end{enumerate}
\end{lem}

\begin{proof}
(i) follows immediately from \Cref{def:i cover}(i). Since $J\subseteq J'$, every element of $G_i^+(J)$ is divisible by some element of $G_i^+(J')$. So (ii) is a consequence of (i). From (ii) we get (iii). For (iv) it suffices to show that any $i$-cover $C$ of $J$ is also an $i$-cover of $\sqrt{J}$. Let $v\in G_i^+(\sqrt{J})$. Then $v^k$ is divisible by an element $u\in G_i^+(J)$ for some $k\ge 1$. Since $C$ covers $u$, it also covers $v^k$. Hence, $C$ covers $v$, as desired. Note that one can also prove (iv) by using \Cref{min_primes} and the fact that $\langle G_i^+(\sqrt{J})\rangle=\sqrt{\langle G_i^+(J)}\rangle.$
\end{proof}

\begin{lem}
 \label{i cover chain}
 Let $\Icc=(I_n)_{n\ge 1}$ be an $\Inc^i$-invariant chain of monomial ideals. Then
 \[
  \gamma_i(I_n)=\gamma_i(I_{n+1})\quad\text{for all }\  n\ge \ind^i(\Icc).
 \]
\end{lem}

\begin{proof} 
It suffices to consider proper ideals. 
We have $\gamma_i(I_n)\le\gamma_i(I_{n+1})$ for all $n\ge 1$ by \Cref{i cover properties}(iii). For $n\ge \ind^i(\Icc)$ one has
$
I_{n+1}=\langle\Inc^i_{n,n+1}(I_n)\rangle,
$
which implies
\[
  G_i^+(I_{n+1})\subseteq \Inc^i_{n,n+1}(G_i^+(I_n)).
\]
 It follows that any $i$-cover of $I_n$ is also an $i$-cover of $I_{n+1}$, because the action of $\Inc^i$ keeps the first index of the variables unchanged. Therefore, $\gamma_i(I_n)\ge\gamma_i(I_{n+1})$, and hence $\gamma_i(I_n)=\gamma_i(I_{n+1})$.
\end{proof}

We set
\[
 \gamma_i(\Icc)= \gamma_i(I_n)\quad\text{for some }\  n\ge \ind^i(\Icc). 
\]
This is well-defined by \Cref{i cover chain}, and moreover, according to \Cref{min_primes}, $\gamma_i(\Icc)$ can be determined by the minimal primes of the ideal $\langle G_i^+(I_n)\rangle$ for any $n\ge \ind^i(\Icc)$.

For convenience in stating and proving the next result we will make use of the following convention:

\begin{convention}
\label{conven}
 The codimension of the unit ideal in the ring $R_n$ is set to be $\infty$ for every $n\ge1$.
\end{convention}

The main result of this section is:

\begin{thm}
 \label{codim}
 Let $\Icc=(I_n)_{n\ge 1}$ be an $\Inc^i$-invariant chain of monomial ideals. Then there exists an integer $D(\Icc)$ such that
 \[
  \codim I_n=\gamma_i(\Icc)n + D(\Icc) \quad \text{for }\ n\gg 0.
 \]
\end{thm}

The argument requires some further preparations. The following observation says that it suffices to prove the theorem for chains of squarefree monomial ideals.

\begin{lem}
 \label{radical}
 Let $\Icc=(I_n)_{n\ge 1}$ be an $\Inc^i$-invariant chain of monomial ideals. Then the chain $\sqrt{\Icc}=(\sqrt{I_n})_{n\ge 1}$ is also $\Inc^i$-invariant with
 $$
  \gamma_i(\sqrt{\Icc})=\gamma_i(\Icc).
 $$
\end{lem}

\begin{proof} 
Again it suffices to consider proper ideals.  
 Let $n\ge m\ge1$, $\pi\in \Inc^i_{m,n}$, and consider any monomial  $u\in \sqrt{I_m}$. Let $k\ge1$ be such that $u^k\in I_m$. Then
 \[
  \pi(u)^k=\pi(u^k)\in\pi(I_m)\subseteq I_n.
 \]
Thus, $\pi(u)\in \sqrt{I_n}$, and so the chain $\sqrt{\Icc}$ is $\Inc^i$-invariant. The equality \[\gamma_i(\sqrt{\Icc})=\gamma_i(\Icc)\] 
follows from \Cref{i cover properties}(iv).
\end{proof}

\begin{lem}
	\label{variable}
	Let $J \subseteq R_n$ be a squarefree monomial ideal and $x$ a variable of $R_n$. Then
	\[
	\codim J =\min \{\codim \langle(J : x),x\rangle -  1,\; \codim \langle J, x\rangle\}.
	\]
\end{lem}

\begin{proof}
	If  $J=R_n$ or $J:x=R_n$, then the formula is true according to \Cref{conven}. If $J$ and $J:x$ are both proper ideals of $R_n$, then it is apparent that
	\[
	\codim J =\min \{\codim \langle J : x\rangle,\; \codim \langle J, x\rangle\}.
	\]
	Since $J$ is squarefree, $x$ is a non-zero-divisor on $R_n/\langle J: x\rangle$. This gives
	\[
	\codim \langle J : x\rangle = \codim \langle(J : x),x\rangle -  1,
	\]
	which yields the desired conclusion.
\end{proof}

The next lemma plays a crucial role in the proof of \Cref{codim}.

\begin{lem}
	\label{linearform_manysteps}
	Let $\Icc=(I_n)_{n\ge 1}$ be an $\Inc^i$-invariant chain of monomial ideals. Fix an integer $r\ge \ind^i(\Icc)$.
	For each ${\bf e}=(e_1,\ldots,e_c)\in \mathbb{Z}^c_{\ge0}$, define the chain $\Icc_{\bf e}=(I_{{\bf e},n})_{n\ge 1}$ as in \Cref{e chain}. 
	Then the following statements hold:
	\begin{enumerate}
	\item
	If $q(I_{{\bf e},r+1})= q(I_r)$, then
	\[
	 \codim I_{{\bf e},n+1}= \codim I_n+c\quad \text{for all }\ n\ge r.
	\]
	
	\item
	 If $I_n$ is a squarefree ideal, then
	\[
	\codim {I_n} =\min \{\codim I_{{\bf e},n}-|{\bf e}|\mid {\bf e}\in \{0,1\}^c\},
	\]
	 where $|{\bf e}|=e_1+\cdots+e_c$.
	
	\item
	For all ${\bf e}\in \mathbb{Z}^c_{\ge0}$ one has $\gamma_i(\Icc)\le \gamma_{i+1}(\Icc_{\bf e})$, with equality if $q(I_{{\bf e},r+1})= q(I_r)$.
	
	\item
	Assume $I_r\ne R_r$. If ${\bf e}\in \{0,1\}^c$ and $q(I_{{\bf e},r+1})= q(I_r)$, then $\gamma_i(\Icc)\le c-|{\bf e}|.$
	
	\item
	Assume $I_r\ne R_r$. Set
	\[
	 E_1=\{ {\bf e}\in \{0,1\}^c\mid q(I_{{\bf e},r+1})<q(I_r)\}\ \text{and}\
         E_2=\{ {\bf e}\in \{0,1\}^c\mid q(I_{{\bf e},r+1})=q(I_r)\}.
	\]
	Then
	\begin{equation}
	 \label{gamma}
	 \gamma_i(\Icc)=\min\{\min\{\gamma_{i+1}(\Icc_{\bf e})\mid {\bf e}\in E_1\},\; \min\{c-|{\bf e}|\mid {\bf e}\in E_2\}\}.
	\end{equation}
	
	\end{enumerate}
\end{lem}

\begin{proof}
	(i) From \Cref{e chain}(ii) one gets the isomorphisms
	\[
	R_{n+1}/I_{{\bf e},n+1} \cong R_n/I_n \quad\text{for all $n\ge r$},
	\]
        which yield the assertion.
	
	(ii) Using \Cref{variable}, the assertion follows by induction on $c$.

        (iii) Let $\sigma_i$ be the $i$-shift defined preceding \Cref{e chain}. Since $\sigma_i\in\Inc^i_{n,n+1}$, one has
        \[
        I_{{\bf e},n+1}\supseteq\langle I_{n+1},x_{1,i+1},\ldots,x_{c,i+1}\rangle\supseteq\langle\sigma_i(I_n),x_{1,i+1},\ldots,x_{c,i+1}\rangle\ \text{ for all }\ n\ge 1.
        \]
        By \Cref{i cover properties}(iii), this gives
        \begin{align*}
         \gamma_{i+1}(I_{{\bf e},n+1})&\ge \gamma_{i+1}(\langle\sigma_i(I_n),x_{1,i+1},\ldots,x_{c,i+1}\rangle)
                                    =  \gamma_{i+1}(\langle\sigma_i(I_n)\rangle)=  \gamma_i(I_n).
        \end{align*}
        The last equality follows from the definition of $\sigma_i$. The only inequality in the above equation becomes an equality if $n\ge r$ and $q(I_{{\bf e},r+1})= q (I_r)$, by \Cref{e chain}(ii).

        (iv) Set $C=\{k\in[c]\mid e_k=0\}$. By (iii), it suffices to show that $C$ is an $(i+1)$-cover of $I_{{\bf e},n+1}$ for $n\ge r$. Assume the contrary. Then there exists a monomial $u\in G_{i+1}^+(I_{{\bf e},n+1})$ of lowest degree which  is not divisible by $x_{k,j}$ for any $k\in C$ and $j\ge 1.$ If there is more than one such monomial we choose $u$ with $\min(u)$ as small as possible. Note that $\min(u)\ge i+2$ since $u\in G_{i+1}^+(I_{{\bf e},n+1})$. By \Cref{e chain}(ii),
        \[
         I_{{\bf e},n+1}=\langle\sigma_i(I_n),x_{1,i+1},\ldots,x_{c,i+1}\rangle.
        \]
        It follows that $u\in \sigma_i(I_n)$. So $u=\sigma_i(v)$ for some $v\in I_n$. By definition of $\sigma_i$ one has
        \[
         \min(v)=\min(u)-1\ge i+1.
        \]        
By the choice of $u$, the monomial $v$ is a minimal generator of  $I_{{\bf e},n+1}$.
If $\min(v)\ge i+2$, then $v\in G_{i+1}^+(I_{{\bf e},n+1})$. Since $\min(v)<\min(u)$, this is a contradiction to $\min (u)$ being least possible among all the monomials in $G_{i+1}^+(I_{{\bf e},n+1})$ of lowest degree which are not divisible by $x_{k,j}$ for any $k\in C$ and $j\ge 1.$ 
        Hence, $\min(v)= i+1$, and we may write $v=x_{l_1,i+1}\cdots x_{l_s,i+1}v'$, where $v'\in R_n$ with $\min(v')> i+1$. Since $u=\sigma_i(v)$ is not divisible by any $x_{k,j}$ with $k\in C$, we must have $l_1,\dots,l_s\notin C$. Thus, $e_{l_1}=\cdots=e_{l_s}=1$, and so
        \[
         v'\in I_n:x_{l_1,i+1}\cdots x_{l_s,i+1}\subseteq I_{{\bf e},n}.
        \]
        Since $\Icc_{\bf e}$ is $\Inc^{i+1}$-invariant (see \Cref{e chain}) and $\sigma_{i+1}\in\Inc^{i+1}_{n,n+1}$, we obtain  $\sigma_{i+1}(v')\in I_{{\bf e},n+1}$. As $\min(v')> i+1$, one has $\sigma_i(v')=\sigma_{i+1}(v')$ . Thus, $\sigma_{i}(v')\in I_{{\bf e},n+1}$. But this contradicts our assumption that $u=\sigma_i(v)=x_{l_1,i+2}\cdots x_{l_s,i+2}\sigma_i(v')$ is a minimal generator of $I_{{\bf e},n+1}$.

        (v) Let $\gamma$ denote the right-hand side of \Cref{gamma}. From (iii) and (iv) it follows that $\gamma_i(\Icc)\le \gamma$. For the reverse inequality it suffices to find a tuple ${\bf e}\in \{0,1\}^c$ with
        \[
         \gamma_{i+1}(\Icc_{\bf e})\le \gamma_i(\Icc)=c-|{\bf e}|.
        \]
        Let $n\ge r+1$ and $C$ an $i$-cover of $I_n$ with $\gamma_i(I_n)=|C|$. 
        Consider ${\bf e}=(e_1,\dots,e_c)$ with
        \[
         e_k=\begin{cases}
              0&\text{if }\ k\in C,\\
              1&\text{if }\ k\notin C.
             \end{cases}
        \]
        Then it is clear that $\gamma_i(I_n)=|C|=c-|{\bf e}|.$ To complete the proof we will show that $C$ is an $(i+1)$-cover of $I_{{\bf e},n}$. For any $u\in G_{i+1}^+(I_{{\bf e},n})$ one has
        \[
         v=x_{1,i+1}^{e_1}\cdots x_{c,i+1}^{e_c}u=\prod_{k\notin C}x_{k,i+1}u\in I_n.
        \]
        This implies that $v$ has a divisor $v'\in G_{i}^+(I_{n})$. Since $C$ covers $v'$, it covers $v$ as well. It then follows that $C$ must cover $u$. Therefore, $C$ is an $(i+1)$-cover of $I_{{\bf e},n}$, as desired.
\end{proof}

We are now ready to prove \Cref{codim}.

\begin{proof}[Proof of \Cref{codim}]
Using \Cref{radical}, we may assume that $\Icc$ is a chain of squarefree monomial ideals. Let $\mathcal{F}$ denote the family of all $(i,r,\Icc)$, where $i,r\ge 0$ are integers and $\Icc=(I_n)_{n\ge 1}$ is an $\Inc^i$-invariant chain of squarefree monomial ideals with $\ind^i(\Icc)\le r$. 

Following the idea of the proofs of \cite[Theorem 6.2]{NR17} and \cite[Theorem 6.2]{LNNR}, we argue by induction on $q= q(I_r)$ that for any $(i,r,\Icc)\in \mathcal{F}$ one has that
\[
\codim I_{n+1} = \codim I_n + \gamma_i(\Icc) \quad \text{whenever }\ n\gg 0.
\]

If $q=0$, then $I_r=R_r$, and so $I_n=R_n$ for every $n\ge r$. By \Cref{conven}, this means that $\codim I_n=\infty$ for $n\ge r$, and the desired conclusion holds according to our convention in \Cref{def:i cover}.

Now assume $q\ge 1$. For each ${\bf e}\in \{0,1\}^c$, we consider the chain $\Icc_{\bf e}=(I_{{\bf e},n})_{n\ge 1}$ as in \Cref{e chain}.
By this lemma, $\Icc_{\bf e}$ is an $\Inc^{i+1}$-invariant chain with $\ind^{i+1}(\Icc_{\bf e})\le r+1$. Write $\{0,1\}^c=E_1\cup E_2$ with
\[
E_1=\{ {\bf e}\in \{0,1\}^c\mid q(I_{{\bf e},r+1})<q\}\quad \text{and}\quad
E_2=\{ {\bf e}\in \{0,1\}^c\mid q(I_{{\bf e},r+1})=q\}.
\]
If ${\bf e}\in E_2$, then \Cref{linearform_manysteps}(i) gives
\[
\codim I_{{\bf e},n}=\codim I_{n-1} +c \quad \text{for all }\ n\ge r+1.
\]
Hence, for all $n \ge r+1$ it follows from \Cref{linearform_manysteps}(ii) that
\begin{align}
    \label{eq:compar}
\codim I_n & = \min\big\{\min\{\codim I_{{\bf e},n}-|{\bf e}| \mid  {\bf e}\in E_1\},\ \min\{\codim I_{{\bf e},n}-|{\bf e}|\mid {\bf e}\in E_2\}\big\}  \nonumber  \\
               & =  \min\big\{\min\{\codim I_{{\bf e},n}-|{\bf e}|\mid {\bf e}\in E_1\},\ \codim I_{n-1}+ \min\{c-|{\bf e}|\mid {\bf e}\in E_2\}\big\}.
\end{align}

For ${\bf e} \in E_1$, the induction hypothesis applied to $(i+1,r+1,\Icc_{\bf e})\in\mathcal{F}$ yields the existence of an integer $N(\Icc_{\bf e})\ge r+1$ such that
\begin{equation}
    \label{eq:use IH}
\codim I_{{\bf e}, n+1} = \codim I_{{\bf e}, n} + \gamma_{i+1}(\Icc_{\bf e}) \quad \text{ whenever } n \ge N(\Icc_{\bf e}).
\end{equation}
Set
\[
N = \max \{N(\Icc_{\bf e})  \mid {\bf e} \in E_1\}.
\]
We will show 
\begin{equation}
 \label{eq:induction}
 \codim I_{n+1} = \codim I_n + \gamma_i(\Icc) \quad \text{whenever }\ n\ge N.
\end{equation}
Indeed, by \Cref{linearform_manysteps}(v) 
\begin{equation}
	 \label{gamma2}
	 \gamma_i(\Icc)=\min\{\min\{\gamma_{i+1}(\Icc_{\bf e})\mid {\bf e}\in E_1\},\; \min\{c-|{\bf e}|\mid {\bf e}\in E_2\}\}.
\end{equation}
 If $n \ge N$ and ${\bf e}  \in E_1$, then
\begin{align*}
\codim I_{{\bf e},n+1}- |\be| & = \codim I_{\be, n} + \gamma_{i+1}(\Icc_{\bf e}) - |\be | \\
& \ge \codim I_n + \gamma_{i+1}(\Icc_{\bf e}) &&\text{by \Cref{eq:compar} } \\
& \ge \codim I_n + \gamma_i(\Icc) &&\text{by \Cref{gamma2}. }
\end{align*}
Combined with \Cref{eq:compar,gamma2}, this implies
\begin{equation}
\begin{aligned}
    \label{eq:upper bound}
\codim I_{n+1}  & =  \min\big\{\min\{\codim I_{{\bf e},n+1}-|{\bf e}|\mid {\bf e}\in E_1\},\ \codim I_{n}+ \min\{c-|{\bf e}|\mid {\bf e}\in E_2\}\big\}\\
&\ge \min\big\{\min\{\codim I_n + \gamma_i(\Icc)\mid {\bf e}\in E_1\},\ \codim I_{n}+ \min\{c-|{\bf e}|\mid {\bf e}\in E_2\}\big\}\\
&\ge\codim I_n + \gamma_i(\Icc) \quad \text{ if } n \ge N.
\end{aligned}
\end{equation}
Moreover, \Cref{eq:compar} gives
\[
\codim I_{n+1} \le \codim I_n + \min \{ c-|\be |  \mid  \be \in E_2\}  \quad \text{ if } n \ge N.
\]
This together with Inequality~\eqref{eq:upper bound} yields \Cref{eq:induction} if $\gamma_i(\Icc) = \min \{c- |\be |  \mid  \be \in E_2\}$.

Thus, it remains to consider the case $\gamma_i(\Icc) < \min \{c- |\be |  \mid  \be \in E_2\}$. Suppose \Cref{eq:induction} is not true. Taking into account Inequality~\eqref{eq:upper bound}, this means that, for any $n_0 \ge N$, there is some $n > n_0$ with
\[
\codim I_{n+1} > \codim I_n + \gamma_i(\Icc).
\]
We use this to define an increasing sequence $(n_j)_{j \in \N}$ of integers:  set $n_0 = N$ and, for $j \ge 1$, let
$n_j$ be the least integer $n > n_{j-1}$ with $\codim I_{n+1} > \codim I_n + \gamma_i(\Icc)$. Thus, we obtain for every $j \ge 1$,
\begin{equation}
	\label{eq:ineq}
\codim I_{n_j+1} \ge  \codim I_N  + (n_j + 1 - N) \gamma_i(\Icc) +j.
\end{equation}
Our assumption $\gamma_i(\Icc) < \min \{c- |\be |  \mid  \be \in E_2\}$ allows us to fix some $\be_0 \in E_1$ such that $\gamma_i(\Icc) = \gamma_{i+1}(\Icc_{{\bf e}_0})$. Let $j$ be an integer with $j > \codim I_{\be_0, N}-\codim I_N- | \be_0 |$, i.e.,
\[
\codim I_N  + (n_j + 1 - N) \gamma_i(\Icc) + j > \codim I_{\be_0,N}  + (n_j + 1 - N) \gamma_i(\Icc)  - | \be_0 |.
\]
Combining this with Inequality~\eqref{eq:ineq} one gets
\begin{align*}
\codim I_{n_j + 1} & >  \codim I_{\be_0,N}  + (n_j + 1 - N) \gamma_i(\Icc)  - | \be_0 |\\
& =  \codim I_{\be_0,N}  + (n_j + 1 - N) \gamma_{i+1}(\Icc_{{\bf e}_0})  - | \be_0 | \\
& = \codim I_{\be_0, n_j+1} - | \be_0 |  \quad \text{ by \Cref{eq:use IH}.}
\end{align*}
However, this contradicts \Cref{eq:compar}. The proof is complete.
\end{proof}

As a consequence, we obtain an explicit and more general version of the first part of \cite[Theorem 7.10]{NR17}. We use the convention that the dimension of a zero module is $-\infty$. 

\begin{cor}
 \label{dim}
 Let $\Icc=(I_n)_{n\ge 1}$ be an $\Inc^i$-invariant chain of ideals, and let $\le$ be any monomial order respecting $\Inc^i$. 
 Then one has 
  \[
  \dim R_n/I_n=A (\Icc) n + B (\Icc)\quad \text{for }\ n\gg 0,  
 \]
 where $A (\Icc) = c-\gamma_i(\ini_{\le}(\Icc))$ and $B (\Icc) = - D (\ini_{\le}(\Icc))$. In particular, the integers $\gamma_i (\Icc) = \gamma_i(\ini_{\le}(\Icc))$ and $D (\Icc) = D (\ini_{\le}(\Icc))$ are independent of the choice of $\le$.  
\end{cor}

\begin{proof}
 Since $\codim I_n=\codim (\ini_{\le}(I_n))$ for all $n\ge 1$ (see, e.g., \cite[Proposition 3.1(a)]{BC}), the result follows from \Cref{lem_initial_filtration} and \Cref{codim}.  
 \end{proof}
 
 \begin{remark}
  Let $\Icc=(I_n)_{n\ge 1}$ be a chain of ideals. If this chain is $\Inc^i$-invariant for some $i\ge 0$, then by definition, it is also $\Inc^j$-invariant for all $j\ge i$. So \Cref{dim} implies that $\gamma_i (\Icc)=\gamma_j (\Icc)$ for all $j\ge i$. Hence, $\gamma_i (\Icc)$ is independent of the choice of $i$ such that $\Icc$ is an $\Inc^i$-invariant chain.
 \end{remark}


\section{Projective dimension up to symmetry}\label{sec4}

In this section we provide evidence for \Cref{conj}. First, recall that an ideal $J\subseteq R_n$ is \emph{perfect} if $\codim J=\pd (R_n/J)$. \Cref{codim} thus implies:

\begin{prop}
 \label{CM}
 If $\Icc=(I_n)_{n\ge 1}$ is an $\Inc^i$-invariant chain of  ideals such that $I_n$ is perfect for all $n\gg0$, then \Cref{conj} is true for $\Icc$.
\end{prop}

\begin{example}
 Chains that satisfy the assumption of \Cref{CM} include the following interesting ones:
 \begin{enumerate}
  \item
  $\Icc$ is generated by one monomial orbit; see \cite[Corollary 2.2]{GN}.

  \item
  $I_r$ is an Artinian ideal in $R_r$ for some $r\ge \ind^i(\Icc)$. 

  \item
  $I_n$ is generated by the $t$-minors of the $c \times n$ matrix with entries being the variables of $R_n$ for all $n\ge 1$, where $t \le c$ is a fixed integer; see, e.g., \cite[Theorem 7.3.1]{BH}. 
  
  \item 
  More generally, 
  $\Icc$ is a chain of graded ideals such that $R_n/I_n$ is Cohen-Macaulay for $n \gg 0$; see, e.g., \cite[Corollary 2.2.15]{BH}. 
 \end{enumerate}
 \end{example}

 Using the notation of \Cref{dim}, it is not hard to give linear bounds for $\pd (R_n/I_n)$:

\begin{prop}
 \label{pd both bounds}
 Let $\Icc=(I_n)_{n\ge 1}$ be an $\Inc^i$-invariant chain of  proper ideals. Then one has
  \[
  cn\ge \pd (R_n/I_n) \ge \gamma_i (\Icc) n+D(\Icc) \quad\text{for }\ n\gg0. 
 \]       
\end{prop}

\begin{proof}
 The upper bound is Hilbert's Syzygy theorem. For the lower bound,  note that $I_n\ne R_n$ for all $n\ge 1$ by assumption. By using  \Cref{dim} and the estimate 
 \[
  \pd (R_n/I_n) \ge \codim I_n \quad \text{for all }\ n\ge 1
 \]
 (see, e.g.,  \cite[Corollary 16.12]{BV}), the desired conclusion follows.
\end{proof}

In the remaining part of this section we focus on chains of proper monomial ideals. Our main results are lower  linear bounds for $\pd I_n = \pd (R_n/I_n) -1$ that improve the bound given in \Cref {pd both bounds}.

Let $\Icc=(I_n)_{n\ge 1}$ be an $\Inc^i$-invariant chain of proper monomial ideals. Fix an integer $r\ge \ind^i(\Icc)$.  For each ${\bf e}\in \mathbb{Z}^c_{\ge0}$ consider the chain $\Icc_{\bf e}=(I_{{\bf e},n})_{n\ge 1}$ as in \Cref{e chain}. We know that $\Icc_{\bf e}$ is an $\Inc^{i+1}$-invariant chain with $\gamma_i(\Icc)\le \gamma_{i+1}(\Icc_{\bf e})$. Let
\begin{equation}
\label{E}
 E(\Icc)=\{{\bf e}\in \mathbb{Z}^c_{\ge0}\mid I_{{\bf e},r+1}\ne R_{r+1}\}.
\end{equation}
Define
\[
 \gamma_{i+1}^{\max}(\Icc)=\max\{\gamma_{i+1}(\Icc_{\bf e})\mid {\bf e}\in E(\Icc)\}.
\]
Similarly, for each ${\bf e}\in E(\Icc)$ and ${\bf e}'=(e_1',\dots,e_c')\in \mathbb{Z}^c_{\ge0}$ one can build an $\Inc^{i+2}$-invariant chain $\Icc_{{\bf e},{\bf e}'}=(I_{{\bf e},{\bf e}',n})_{n\ge 1}$ with
\[I_{{\bf e},{\bf e}',n}=
	\langle (I_{{\bf e},n}:x_{1,i+2}^{e_1'}\cdots x_{c,i+2}^{e_c'}),x_{1,i+2},\ldots,x_{c,i+2}\rangle \quad\text{for all }\ n\ge 1.
\]
Then define
\[
 E(\Icc_{\bf e})=\{{\bf e}'\in \mathbb{Z}^c_{\ge0}\mid I_{{\bf e},{\bf e}',r+2}\ne R_{r+2}\} 
\]
and
\[
 \gamma_{i+2}^{\max}(\Icc)=\max\{\gamma_{i+2}(\Icc_{{\bf e},{\bf e}'})\mid {\bf e}\in E(\Icc),\; {\bf e}'\in E(\Icc_{\bf e})\}.
\]
Repeating this construction we obtain a non-decreasing sequence of integers 
\[
 \gamma_i(\Icc)\le \gamma_{i+1}^{\max}(\Icc)\le \gamma_{i+2}^{\max}(\Icc)\le \cdots\le c.
\]
Let $\Gamma_i(\Icc)$ denote the limit of this sequence:
\[
 \Gamma_i(\Icc)=\max\{\gamma_{i+k}^{\max}(\Icc)\mid k\ge 1\}.
\]
It is obvious that
\[
 \Gamma_i(\Icc)\ge \gamma_i(\Icc).
\]
This inequality is strict in general, as illustrated below.

\begin{example}
 Let $r\ge i+2$ and consider the chain $\Icc=(I_n)_{n\ge 1}$ with
 \[
  I_n=\begin{cases} \langle 0\rangle&\text{if }\ n<r,\\
                    \langle x_{1,i+1}x_{1,i+2},\; x_{1,i+1}x_{2,i+2},\dots,\; x_{1,i+1}x_{c,i+2}\rangle &\text{if }\ n=r,\\
                    \langle\Inc^i_{r,n}(I_r)\rangle &\text{if }\ n>r.
      \end{cases}
 \]
 Then $\ind^i(\Icc)=r$. One has $\gamma_i(\Icc)=1$, since $C=\{1\}$ is an $i$-cover of $I_r$. Now for ${\bf e}=(1,0,\dots,0)$ it is easily seen that
 \[
  I_{{\bf e},r+1}=\langle x_{1,i+3},\dots,\; x_{c,i+3},\; x_{1,i+2},\dots,\; x_{c,i+2}, \; x_{1,i+1},\dots,\; x_{c,i+1}\rangle.
 \]
 This gives $\gamma_{i+1}(\Icc_{\bf e})=c$, which implies $\Gamma_i(\Icc)=c$, because $\gamma_{i+1}(\Icc_{\bf e})\le \Gamma_i(\Icc)\le c$. Hence,  if $c>1$ then 
 \[
 c = \Gamma_i(\Icc)> \gamma_i(\Icc) =1, 
 \]
 and the difference is as large as possible. 
\end{example}

\Cref{min_primes} provides an efficient way to compute $\gamma_i(\Icc)$. It would be interesting to know a similar result for $\Gamma_i(\Icc)$.

\begin{question}
 Is there an alternate way to determine $\Gamma_i(\Icc)$? In particular, with notation of \Cref{min_primes}, is it true that
 \[
  \Gamma_i(\Icc)\in \min\{\#C\mid C\in \mathcal{C}_i(I_n)\} \quad\text{for some }n\ge \ind^i(\Icc)?
 \]
\end{question}

Let us now improve the lower bound in \Cref{pd both bounds} for chains of monomial ideals.

\begin{thm}
 \label{pd_bound}
 Let $\Icc=(I_n)_{n\ge 1}$ be an $\Inc^i$-invariant chain of  proper monomial ideals. Then there exists an integer $\tilde{D}(\Icc)$ such that
 \[
   \pd I_n\ge \Gamma_i(\Icc)n+\tilde{D} (\Icc)\quad\text{for }\ n\gg0.
 \]
\end{thm}

This result gives the following necessary condition for eventual Cohen-Macaulayness of $\Inc^i$-invariant chains of monomial ideals:

\begin{cor}
      \label{cor:CM condition}
Let $\Icc=(I_n)_{n\ge 1}$ be an $\Inc^i$-invariant chain of proper ideals.  If there is a monomial order $\le$ respecting $\Inc^i$ and such that  $R_n/\ini_{\le} (I_n)$ is Cohen-Macaulay for $n \gg 0$, then
\[
\gamma_i (\Icc) = \Gamma_i(\ini_{\le} (\Icc)).
\]
\end{cor}

\begin{proof}
Since $R_n/\ini_{\le} (I_n)$ is Cohen-Macaulay if and only if $\codim \ini_{\le} (I_n) = \pd R_n/\ini_{\le} (I_n)$, the result follows by combining Theorems \ref{codim} and \ref{pd_bound}.
\end{proof}

To prove \Cref{pd_bound} we need some auxiliary results. 

\begin{lem}
	\label{pd variable}
	Let $J \subseteq R_n$ be a monomial ideal and $x$ a variable of $R_n$. Then
	\[
	\max \{\pd \langle J : x\rangle,\; \pd \langle J, x\rangle-1\}\le\pd J \in \{\pd \langle J : x\rangle,\; \pd \langle J, x\rangle\}.
	\]
	Moreover, if $d\ge 1$ is an integer such that $J:x^{d}=J:x^{d+1}$, then
	\[
	\max \{\pd \langle(J : x^k),x\rangle\mid 0\le k\le d\}-1\le\pd J \in \{\pd \langle(J : x^d),x\rangle-1,\; \pd \langle(J:x^k), x\rangle\mid 0\le k< d\}.
	\]
\end{lem} 

\begin{proof}
 The containment in the first assertion follows from \cite[Corollary 3.3(i)]{CH+} and the Auslander-Buchsbaum formula, whereas the lower bound follows from the containment and the exact sequence 
\[
0 \to (R_n/\langle J : x\rangle)(-1) \to R_n/J \to R_n/\langle J, x\rangle \to 0. 
\]
Now applying the first assertion to the ideals $J : x^k$ for $k\ge0$ we get
\begin{equation}
\label{eq:lm47}
 \max \{\pd \langle J : x^{k+1}\rangle,\; \pd \langle (J:x^k), x\rangle-1\}\le\pd \langle J:x^k\rangle \in \{\pd \langle J : x^{k+1}\rangle,\; \pd \langle (J:x^k), x\rangle\}.
\end{equation}
Since $J:x^{d}=J:x^{d+1}$, $x$ is a non-zero-divisor on $R_n/\langle J:x^d\rangle$, which gives
\[
 \pd \langle J : x^{d+1}\rangle=\pd \langle J:x^d\rangle=\pd \langle (J:x^d), x\rangle-1.
\]
Combining this with \eqref{eq:lm47} for $k=0,\dots,d$ yields the second assertion.
\end{proof}

\begin{lem}
 \label{pd chain}
 Let $\Icc=(I_n)_{n\ge 1}$ be an $\Inc^i$-invariant chain of proper monomial ideals. For each ${\bf e}\in \mathbb{Z}^c_{\ge0}$ consider the chain $\Icc_{\bf e}=(I_{{\bf e},n})_{n\ge 1}$ as in \Cref{e chain}. Let $E(\Icc)$ be defined as in \Cref{E}. Then
 \[
  \max \{\pd I_{{\bf e},n}\mid {\bf e}\in E(\Icc)\}-c\le\pd I_n \le \max \{\pd I_{{\bf e},n}\mid {\bf e}\in E(\Icc)\}\quad\text{for all }\ n\ge r+1.
 \]
\end{lem}

\begin{proof}
 By induction on $c$ it suffices to consider the case $c=1$. Let $n\ge r+1$ and choose an integer $d\ge1$ such that $I_n:x_{1,i+1}^d=I_n:x_{1,i+1}^{d+1}$. Then \Cref{pd variable} gives
 \begin{equation}
 \begin{aligned}
  \max &\{\pd I_{{\bf e},n}\mid 0\le {\bf e}\le d\}-1=\max \{\pd \langle(I_n:x_{1,i+1}^{\bf e}),x_{1,i+1}\rangle\mid 0\le {\bf e}\le d\}-1\\
  &\le\pd I_n \le \max\{\pd \langle(I_n:x_{1,i+1}^{\bf e}),x_{1,i+1}\rangle\mid 0\le {\bf e}\le d\}=\max \{\pd I_{{\bf e},n}\mid 0\le {\bf e}\le d\}.
  \end{aligned}
 \end{equation}
 So to complete the proof we need to show that
 \[
  \max\{\pd I_{{\bf e},n}\mid 0\le {\bf e}\le d\}=\max\{\pd I_{{\bf e},n}\mid {\bf e}\in E(\Icc)\}.
 \]
Note that $I_{{\bf e},n}=I_{d,n}$ for all ${\bf e}\ge d$ since $I_n:x_{1,i+1}^d=I_n:x_{1,i+1}^{d+1}$. One the other hand, $I_{{\bf e},n}=R_n$ has projective dimension $0$ for all ${\bf e}\in \Z_{\ge0}\setminus E(\Icc)$. Therefore,
\[
  \max\{\pd I_{{\bf e},n}\mid 0\le {\bf e}\le d\}=\max\{\pd I_{{\bf e},n}\mid {\bf e}\in \Z_{\ge0}\}=\max\{\pd I_{{\bf e},n}\mid {\bf e}\in E(\Icc)\}.\qedhere
\]
\end{proof}

Now we prove \Cref{pd_bound}.

\begin{proof}[Proof of \Cref{pd_bound}]
 Let $k\ge 1$ be such that $\Gamma_i(\Icc)=\gamma_{i+k}^{\max}(\Icc)$. Applying \Cref{pd chain} iteratively we obtain
 \[
  \pd I_n\ge \max \{\pd I_{{\bf e}_1,\dots,{\bf e}_k,n}\mid {\bf e}_1\in E(\Icc),\dots,{\bf e}_k\in E(\Icc_{{\bf e}_1,\dots,{\bf e}_{k-1}})\}-kc
 \]
 for all $n\gg0$. By \Cref{codim},
 \[
  \pd I_{{\bf e}_1,\dots,{\bf e}_k,n}\ge \codim I_{{\bf e}_1,\dots,{\bf e}_k,n} - 1 =\gamma_{i+k}(\Icc_{{\bf e}_1,\dots,{\bf e}_{k}})n+D(\Icc_{{\bf e}_1,\dots,{\bf e}_{k}}) -1 \quad \text{for } n\gg0.
 \]
 So if we set
 \[
  \tilde{D} (\Icc)=\min\{D(\Icc_{{\bf e}_1,\dots,{\bf e}_{k}}) -1 \mid {\bf e}_1\in E(\Icc),\dots,{\bf e}_k\in E(\Icc_{{\bf e}_1,\dots,{\bf e}_{k-1}})\}-kc,
 \]
 then it follows that
 \[\begin{aligned}
  \pd I_n&\ge \max \{\gamma_{i+k}(\Icc_{{\bf e}_1,\dots,{\bf e}_{k}})\mid {\bf e}_1\in E(\Icc),\dots,{\bf e}_k\in E(\Icc_{{\bf e}_1,\dots,{\bf e}_{k-1}})\}n+\tilde{D}(\Icc)\\
  &=\gamma_{i+k}^{\max}(\Icc)n+\tilde{D}(\Icc)=\Gamma_i(\Icc)n+\tilde{D}(\Icc) 
   \end{aligned}
 \]
 for all $n\gg0$. 
\end{proof}

Next, we discuss another improvement of the lower bound given in \Cref{pd both bounds}. Assume $r\ge \ind^i(\Icc)$ and let $G(I_r)$ be the minimal set of monomial generators of $I_r$. Evidently, every monomial $u\in G_i(I_r)$ can be uniquely written as $u=u_1u_2$ with $\max(u_1)\le i$, $\min(u_2)> i$. Set
\[
 {G}_{i,1}(I_r)=\{u_1\mid u\in G_i(I_r)\}\quad\text{and}\quad{G}_{i,2}(I_r)=\{u_2\mid u\in G_i(I_r)\}.
\]
Observe that the sets $G_i^-(I_r)$ and ${G}_{i,1}(I_r)$ are fixed under the action of $\Inc^i$, whereas  $G_i^+(I_r)$ and ${G}_{i,2}(I_r)$ usually change. So intuitively, one would expect that the growth of $\pd I_n$ depends on $G_i^+(I_r)$ and ${G}_{i,2}(I_r)$. This will be clarified now.

For a subset $M$ of $G_i(I_r)$, set
\[
 v_M=\prod_{u\in M} u_1\in R_i
\]
and consider the chain 
\[
\Icc:v_M=(I_n:v_M)_{n\ge 1}. 
\] 
According to \Cref{colon chain}, $\Icc:v_M$ is an $\Inc^i$-invariant chain with $\ind^i(\Icc:v_M)\le \ind^i(\Icc)$. Assume that the monomial $v_M$ is not divisible by any element of $G_i^{-}(I_r)$. Then $I_r:v_M\ne R_r$ and $G_i^+(I_r:v_M)$ consists of minimal elements (under divisibility) of $G_i^{+}(I_r)\cup\{u_2\mid u\in M\}$. It follows that
\[
 \gamma_i(\Icc)=\gamma_i(\langle G_i^{+}(I_r)\rangle)\le \gamma_i(\langle G_i^{+}(I_r)\cup\{u_2\mid u\in M\}\rangle)=\gamma_i(\langle G_i^{+}(I_r:v_M)\rangle)=\gamma_i(\Icc:v_M).
\]
Thus, the following lower bound also improves the one in \Cref{pd both bounds}.

\begin{thm}
 \label{pd bound 2}
 Let $\Icc=(I_n)_{n\ge 1}$ be an $\Inc^i$-invariant chain of monomial ideals. Let $r\ge \ind^i(\Icc)$ and denote by $\mathcal{M}$  the set of all subsets $M$ of $G_i(I_r)$ such that $v_M$ is not divisible by any element of $G_i^{-}(I_r)$. Then there exists an integer $\tilde{D} (\Icc)$ such that
 \[
  \pd I_n\ge \max\{\gamma_i(\Icc:v_M)\mid M\in \mathcal{M}\}n+\tilde{D}(\Icc)\quad\text{for }\ n\gg0.
 \]
\end{thm} 

\begin{proof}
Using the lower bound in \Cref {pd variable} repeatedly, one gets $\pd I_n \ge \pd (I_n : v_M)$. 
\end{proof}

Let us briefly compare the bounds in Theorems \ref{pd_bound} and \ref{pd bound 2}. Observe that $\Gamma_i(\Icc)$ only depends on $G_i^{+}(I_r)$, while the bound in \Cref{pd bound 2} depends on $G_i^+(I_r)$ and ${G}_{i,2}(I_r)$ (and also on $G_i^-(I_r)$ and ${G}_{i,1}(I_r)$). So for instance, the bound in \Cref{pd_bound} is potentially better if $G_i(I_r)=\emptyset$, while the one in \Cref{pd bound 2} is potentially better if $G_i^+(I_r)=\emptyset$.

As an immediate consequence of \Cref{pd bound 2}, we obtain the following bound, which depends only on $G_i^+(I_r)$ and ${G}_{i,2}(I_r)$.

\begin{cor} 
  \label{cor:weaker}
 Let $\Icc=(I_n)_{n\ge 1}$ be an $\Inc^i$-invariant chain of monomial ideals with $\ind^i(\Icc)\le r$. Assume that
 \[
 v_{G_i(I_r)}=\prod_{u\in G_i(I_r)} u_1
\]
is not divisible by any element of $G_i^{-}(I_r)$. Then there exists a constant $\tilde{D}(\Icc)$ such that
 \[
  \pd I_n\ge \gamma_i(\langle G_i^{+}(I_r)\cup {G}_{i,2}(I_r)\rangle)n+\tilde{D}(\Icc)\quad\text{for }\ n\gg0.
 \]
\end{cor} 

Note that  the assumption of \Cref{cor:weaker} is satisfied if $G_i^-(I_r)=\emptyset$. 

Similarly to \Cref{cor:CM condition}, one gets from \Cref{pd bound 2} another necessary condition for eventual Cohen-Macaulayness of $\Inc^i$-invariant chains. We only state a version for monomial ideals and leave the more general statement to the interested reader. 

\begin{cor}
     \label{cor:CM obsrruction}
 With assumption as in \Cref{pd bound 2}, if $R_n/I_n$ is Cohen-Macaulay for $n \gg 0$, then
\[
\gamma_i(\Icc) = \max\{\gamma_i(\Icc:v_M)\mid M\in \mathcal{M}\}.
\]
\end{cor}

To conclude this section, we consider the case $c=1$, i.e., there is only one row of variables. The next result shows that \Cref{conj} is ``nearly true'' in this case.

\begin{prop}
 Assume $c=1$. Let $\Icc=(I_n)_{n\ge 1}$ be an $\Inc^i$-invariant chain of proper monomial ideals. Then either $\pd (R_n/I_n)$ is eventually a constant or there exists a nonnegative integer $D$ such that
 \[
  n-D\le \pd (R_n/I_n) \le n\quad \text{for all }\ n\gg0.
 \]
\end{prop}

\begin{proof}
 Let $r\ge \ind^i(\Icc)$. We distinguish three cases:

 \emph{Case 1}: $G_i^+(I_r)=G_i(I_r)=\emptyset$. In this case, $G(I_r)=G_i^-(I_r)$ is fixed under the action of $\Inc^i$. It follows that
 \[
  I_n=\langle\Inc^i_{r,n}(I_r)\rangle_{R_n}=\langle I_r\rangle_{R_n} \quad\text{for all }\ n\ge r.
 \]
Hence, $\pd (R_n/I_n) =\pd (R_r/I_r)$ for all $n\ge r$.

 \emph{Case 2}: $G_i^+(I_r)\ne\emptyset$. Then $\gamma_i(\Icc)\ge 1$. Since $\gamma_i(\Icc)\le c=1$, we must have $\gamma_i(\Icc)=1$. Applying \Cref{pd both bounds} the result follows.

 \emph{Case 3}: $G_i(I_r)\ne\emptyset$. Let $u\in G_i(I_r)$ and write $u=u_1u_2$ with $\max(u_1)\le i$, $\min(u_2)> i$. Since $u$ is a minimal generator of $I_r$, $u_1$ is not divisible by any element of $G_i^{-}(I_r)$. Consider the chain $\Icc:u_1$. One has $\gamma_i(\Icc:u_1)\ge 1$ since $u_2\in G_i^+(I_r:u_1)$.
 So using \Cref{pd both bounds} and \Cref{pd bound 2} concludes the proof.
\end{proof}

\end{document}